%% file: hiu.tex
\documentclass[11pt,table]{amsart}
\newenvironment{nouppercase}{%
  \renewcommand{\uppercasenonmath}[1]{}}{}
\let\myuppercase\uppercase%
\title[\texorpdfstring{\myuppercase{Smoothly slice knots with Alex. pol. 1 and high} $u$}{Smoothly slice knots with Alex. pol. 1 and high u}]{\myuppercase{Smoothly slice knots with Alexander polynomial~1 and high unknotting number}}
\author{Lukas Lewark}
\address{ETH Z\"urich, R\"amistrasse 101, 8092 Z\"urich, Switzerland}
\email{\myemail{llewark@math.ethz.ch}}
\urladdr{\url{https://people.math.ethz.ch/~llewark/}}
\subjclass{57K10, 57K18}
\input{preamble.tex}
\newcommand{\MM}{M\!\!\!M}
\renewcommand{\MM}{m\hspace{-5.8pt}m}
\renewcommand{\MM}{\includegraphics{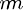}}
\begin{document}
\begin{abstract}
We prove the existence of a smoothly doubly slice, amphicheiral knot with Alexander polynomial~1 and unknotting number~5.
\end{abstract}
\begin{nouppercase}
\maketitle
\end{nouppercase}
\section{Introduction}
The unknotting number $u(K)$ of a knot $K \subset S^3$
is defined as the minimal number of crossing changes required to make $K$ trivial~\cite{zbMATH03026573}.
It is one of the most intuitive and classical measures of the complexity of a knot.
At the same time, the computation of $u(K)$ appears out of reach even for many simple knots~$K$.
One cause for this predicament is that, to put it pointedly, there are no true lower bounds for the unknotting number. That is to say, virtually all known computable lower bounds for $u$
actually are lower bounds for other geometric measures of knot complexity that themselves bound $u$ from below.
The following two instances exemplify this phenomenon and reflect broader patterns among other lower bounds.
Firstly, the $\tau$-invariant from knot Floer homology~\cite{zbMATH02057402,rasmussen2003floer} yields the lower bound $|\tau(K)| \leq u(K)$.
Yet, we also find $|\tau(K)| \leq g_4(K) \leq u(K)$ for the smooth slice genus~$g_4(K)$.
Secondly, the dimension of the first homology space with coefficients in any field of the branched double covering of $S^3$ along $K$ is a lower bound for~$u(K)$~\cite{zbMATH03026573}.
Yet, this quantity is also a lower bound for the minimal number of crossing changes required to reach a knot with Alexander polynomial~1.
In particular, neither of those two lower bounds for $u$ can be used to prove the following main result of this short note.

\begin{theorem}\label{main:thm}
There exists a smoothly slice knot with Alexander polynomial~1
whose unknotting number equals~5.
\end{theorem}
One easily constructs families of smoothly slice, Alexander polynomial~1 knots that, intuition tells us, should have growing unknotting number.
For example, embarking from any non-trivial knot $K$ with Alexander polynomial~1,
the $n$-fold connected sum $J_n$ of $K\#\, {-K}$ is smoothly slice with Alexander polynomial~1, and it is tempting to presume $J_n$ has unknotting number at least~$2n$.
However, it appears quite unclear how to prove this presumption for general knots~$K$ with Alexander polynomial~$1$.
In fact, since the unknotting number is not additive~\cite{arXiv:2506.24088}, it might very well be incorrect.

Scharlemann's theorem that knots with $u=1$ are prime~\cite{zbMATH03921530} implies $u(J_n) \geq 2$.
But to the best knowledge of the author, smoothly slice knots with Alexander polynomial~1 and $u > 2$ have not previously been proven to exist.
Indeed, the author is only aware of a single known type of 
computable lower bounds for the unknotting number
that might be employed for this purpose:
the bounds coming from torsion orders in knot homologies.
In the following section, we will use such a bound to establish \cref{main:thm}.
In addition to the properties claimed in \cref{main:thm},
the knot whose existence we prove will be smoothly doubly slice and amphicheiral.
\section{Proof}
\begin{figure}[b]
\centering
\includegraphics[scale=.7]{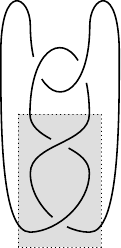}\hspace{2cm}
\includegraphics[scale=.7]{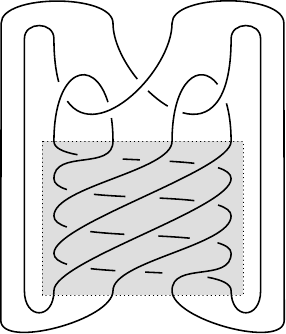}\\[5mm]
\includegraphics[scale=.7]{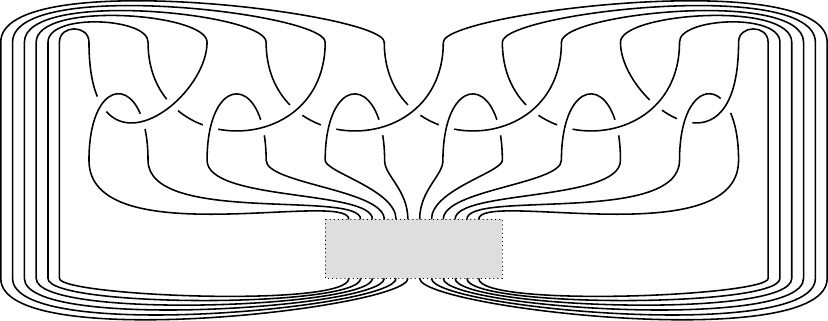}
\caption{The Manolescu--Marengon knots $\MM_1$, $\MM_2$,~$\MM_6$. The gray box for each of the $\MM_{\ell}$ contains a right-handed full twist on $2\ell$ strands. For a diagram of~$\MM_3$, see~\cite{zbMATH07144520}.
$\MM_1$ is the figure-eight knot.
$\MM_2$ is 17nh76 in the tables~\cite{regina}.
}
\label{fig:knots}
\end{figure}
Our starting point is the knot constructed by Manolescu--Marengon~\cite{zbMATH07144520}
to refute the knight move conjecture in Khovanov homology.
That knot can be seen as the third member of an infinite family $\MM_1, \MM_2, \MM_3,\ldots$
of knots, as indicated in \cref{fig:knots}.
Two properties of those knots $\MM_{\ell}$ are relevant here.
Firstly, each of them is evidently related to the unknot by a \emph{null-homologous twist},
i.e.\ by a $(+1)$- or $(-1)$-surgery along the boundary of a disk that has algebraic intersection number zero with the knot (more informally, grabbing an even number of strands of the knot, half of which are oriented in each direction, and tying a full twist into them).
Secondly, the knot Floer homology of those knots conjecturally satisfies the following.
\begin{conjecture}\label{conj:torsion}
The maximal order of $U$-torsion in $\HFK^-(\MM_{\ell})$ equals~$\ell$.
\end{conjecture}%
Let us briefly review the essential aspects of knot Floer homology.
The minus version of knot Floer homology associates
to every knot $K$ a chain complex $\CFK^-(K)$ of free $\mathbb{F}_2[U]$-modules of finite total rank,
well-defined up to homotopy equivalence~(see for example the survey~\cite{zbMATH06696663}).
The chain complex carries the \emph{Maslov grading}~$\mu$
and the \emph{Alexander filtration}~$A$.
The differential~$\partial^-$ lowers the grading $\mu$ by~$1$, and respects the filtration~$A$.
On the associated graded chain complex $\Gr\CFK^-(K)$ of free $\mathbb{F}_2[U]$-modules, the filtration $A$ becomes a grading, which we also denote by~$A$, making $\Gr\CFK^-(K)$ a doubly graded complex. 
Its differential $\Gr\partial^-$ lowers $\mu$ by~$1$ and preserves~$A$;
multiplication by $U$ lowers $\mu$ by~$2$ and $A$ by~$1$.
Because $\mathbb{F}_2[U]$ is a PID, $\Gr\CFK^-(K)$ has,
up to graded homotopy equivalence,
a decomposition as a direct sum of
complexes of rank~1 over~$\mathbb{F}_2[U]$, and complexes of the form
\[
\mu^{i}A^{j}\mathbb{F}_2[U] \xrightarrow{U^n} \mu^{i + 2n - 1}A^{j + n}\mathbb{F}_2[U]
\]
for varying $i,j,n\in\mathbb{Z}$ with $n\geq 1$~(cf.~\cite[Sec.~3]{zbMATH07311838}).
Since the homology of $\Gr\CFK^-(K) \otimes_{\mathbb{F}_2[U]} \mathbb{F}_2[U, U^{-1}]$
is isomorphic to $\mathbb{F}_2[U, U^{-1}]$,
the above direct sum decomposition contains exactly one summand complex of rank~1.
Let us denote the maximal $n$ in the decomposition by~$\mathfrak{t}(K)$,
and let us denote the homology of~$\Gr\CFK^-(K)$,
which consists of (generally non-free) graded $\mathbb{F}_2[U]$-modules, by~$\HFK^-(K)$.
Then $\mathfrak{t}(K)$ equals the maximal order of $U$-torsion in~$\HFK^-(K)$, to which \cref{conj:torsion} refers.
We now provide a partial proof of \cref{conj:torsion}.

\begin{lemma}
For $1\leq \ell \leq 6$, the maximal order of $U$-torsion in $\HFK^-(\MM_{\ell})$ is at least~$\ell$.
\end{lemma}
\begin{proof}
We will rely on Szab\'{o}'s computer program \cite{hfk} to calculate the hat version of knot Floer homology $\widehat{\HFK}(K)$, which is
the homology of $\CFK^-(K) \otimes_{\mathbb{F}_2[U]} \mathbb{F}_2[U]/(U)$.
Note that $\widehat{\HFK}(K)$---a doubly graded vector space over $\mathbb{F}_2$---may be obtained from the direct sum decomposition of $\Gr \CFK^-(K)$ discussed above simply by setting~${U = 0}$. So every summand complex with differential $U^n$
corresponds to a pair of homogeneous homology classes in $\widehat{\HFK}(K)$
with $(\mu,A)$-gradings $(i,j)$ and~$(i + 2n - 1, j + n)$.
One may pick a homogeneous basis for~$\widehat{\HFK}(K)$ where all except one of the basis vectors are paired up like this.
The homologies $\widehat{\HFK}(\MM_{\ell})$ as calculated with Szab\'{o}'s program \cite{hfk}
	are collected in \cref{table:hom1_2,table:hom3_4,table:hom5_6} at the end of the paper (see also~\cite{github:LLewark/hfk-computations}).
For~$1 \leq \ell \leq 6$, one finds:
\begin{align*}
\dim \widehat{\HFK}(\MM_{\ell})_{\mu = A = 0}  & = 2 \quad\text{except for $\ell = 1$, where it is 3}, \\
\dim \widehat{\HFK}(\MM_{\ell})_{\mu = 2A + 1} & = 0 \quad\text{for } {-\ell + 1\leq \mu \leq -1},\\
\dim \widehat{\HFK}(\MM_{\ell})_{\mu = 2A - 1} & = 0 \quad\text{for } \mu \geq 1.
\end{align*}
So one of the two (or three if $\ell = 1$) basis vectors at grading $\mu = A = 0$
can only be paired up by $U^n$ with~$n\geq \ell$.
This proves the claim.
\end{proof}
The maximal order of $U$-torsion in $\HFK^-(K)$ has, amongst others, the following application.
\begin{theorem}[\cite{zbMATH07305772}]\label{thm:utorsion}
The Gordian distance between knots $K$ and~$J$
is bounded from below by~$|\mathfrak{t}(K) - \mathfrak{t}(J)|$.
\end{theorem}
For example, since~$u(\MM_{\ell}) \leq \ell$, as is visible from \cref{fig:knots},
it follows that $\mathfrak{t}(\MM_{\ell}) \leq \ell$.
The maximal order of $U$-torsion
satisfies~\cite{zbMATH07305772,zbMATH07311838}
\[
\mathfrak{t}(K\#J) = \max(\mathfrak{t}(K), \mathfrak{t}(J)), \qquad
\mathfrak{t}(-K)   = \mathfrak{t}(K).
\]
Since the knots $\MM_{\ell}$ do not themselves have Alexander polynomial~1,
we require the following last ingredient.
\begin{proposition}[\cite{FL24}]\label{thm:galg}
Suppose two given knots $K$ and $J$ are related by a null-homologous twist.
Then there exists a knot $J'$ that is $S$-equivalent to $J$ and related
to $K$ by a crossing change.
\end{proposition}
A detailed proof of \cref{thm:galg} will appear in the upcoming paper~\cite{FL24}. Here is a proof sketch.
Let a disk $D$ be given that has algebraic intersection number zero with $K$, such that $(\pm 1)$-surgery along $\partial D$ turns $K$ into~$J$.
There exists a Seifert surface $\Sigma$ for $K$ and a disk $D'$ such that
$[\Sigma\pitchfork D] = {[\Sigma\pitchfork D']} \in H_1(\Sigma, \partial\Sigma; \mathbb{Z})$, and such that $\Sigma\pitchfork D'$ consists of a single proper arc in~$\Sigma$. Then $J'$ is obtained from $K$ by $(\pm 1)$-surgery along~$\partial D'$.

\cref{main:thm} is now inferred by applying the following lemma to~$K = \MM_6$.
\begin{lemma}\label{lem:main}
Let $K$ be a knot with the following two properties for some~$n\geq 2$:
\begin{enumerate}
\item The maximal order of $U$-torsion in $\HFK^-$ of $K$ is at least~$n$.
\item By a single null-homologous twist, $K$ may be transformed into a knot $J$ with Alexander polynomial~1.
\end{enumerate}
Then there is a smoothly doubly slice, Alexander polynomial~1,
amphicheiral knot $L$ with unknotting number at least~$n - 1$.
\end{lemma}
\begin{proof}
By \cref{thm:galg}, there is a knot $J'$ that is related to $K$ by a crossing change,
and $S$-equivalent to~$J$. So $J'$ has Alexander polynomial~1 itself.
By \cref{thm:utorsion}, $\mathfrak{t}(J') \geq n - 1$.
Let~$L = J' \#\, {-J'}$. Then, being the sum of a knot and its concordance inverse, $L$ is smoothly doubly slice~\cite{zbMATH03218554} and amphicheiral. Since $J'$ has Alexander polynomial~1, so has~$L$.
Lastly,
$\mathfrak{t}(L) = \max(\mathfrak{t}(J'), \mathfrak{t}(-J')) = \mathfrak{t}(J') \geq n - 1$.
So by \cref{thm:utorsion}, $u(L) \geq n - 1$.
\end{proof}
\begin{remark}
The proof of \cref{thm:galg} is constructive.
Hence the knots $J'$ and $L$ produced in \cref{lem:main},
and the knot proving \cref{main:thm}, could be made explicit (given sufficient commitment).
\end{remark}
\begin{remark}
One could alternatively employ similar lower bounds for the unknotting number coming from other knot homologies, such as the maximal order of $h$-torsion in Bar-Natan homology~\cite{zbMATH07178864}.
However, current computer programs only calculate Bar-Natan homology of $\MM_{\ell}$ in reasonable time for~$\ell \leq 3$, and---much as knot Floer homology---Bar-Natan homology of $\MM_{\ell}$ appears difficult to compute manually.
\end{remark}

\subsection*{Acknowledgments} The author thanks Marco Marengon, Peter Feller, Sashka Kjuchukova and Sebastian Baader for inspiring conversation.
\bibliographystyle{myamsalpha}
\bibliography{References}

\begin{table}[p]
\centering
\scalebox{.75}{%
\begin{tblr}[b]{hlines={0.7pt, solid}, vlines={0.7pt, solid}, hline{3-Y} = {0.3pt, dashed}, vline{3-Y} = {0.3pt, dashed}, rowspec={*{4}c}, rows={9mm}, columns={9mm}, rowsep=0mm, colsep=0mm}
	& $-1$	& $0$	& $1$	\\
$-1$	& $1$	&  	&  	\\
$0$	&  	& $3$	&  	\\
$1$	&  	&  	& $1$	\\
\end{tblr}}%
\hspace*{1cm}%
\scalebox{.75}{%
\begin{tblr}[b]{hlines={0.7pt, solid}, vlines={0.7pt, solid}, hline{3-Y} = {0.3pt, dashed}, vline{3-Y} = {0.3pt, dashed}, rowspec={*{7}c}, rows={9mm}, columns={9mm}, rowsep=0mm, colsep=0mm}
 	& $-2$	& $-1$	& $0$	& $1$	& $2$	\\
$-4$	& $1$	&  	&  	&  	&  	\\
$-3$	& $1$	& $4$	&  	&  	&  	\\
$-2$	&  	& $2$	& $6$	&  	&  	\\
$-1$	&  	&  \SetCell{red!25}	& $3$	& $4$	&  	\\
$0$	&  	&  	& \SetCell{green!25}$2$	& $2$	& $1$	\\
$1$	&  	&  	&  	& \SetCell{red!25} 	& $1$	\\
\end{tblr}}
\caption{Knot Floer homologies $\widehat{HFK}$ of $\MM_1$ and~$\MM_2$, as computed with~\cite{hfk}.
The Maslov grading is given by the row, and the Alexander grading by the column.
The green cell contains a homology class that must have $U$-torsion order~2{} in~$\HFK^-(\MM_2)$,
because the red cells are empty.
See also~\cite{github:LLewark/hfk-computations}.}%
\label{table:hom1_2}
\end{table}
\begin{table}[p]%
\scalebox{.75}{%
\begin{tblr}[b]{
        hlines={0.7pt, solid}, vlines={0.7pt, solid},
        hline{3-Y} = {0.3pt, dashed}, vline{3-Y} = {0.3pt, dashed},
        rowspec={*{11}c}, rows={9mm}, columns={9mm},
        rowsep=0mm, colsep=0mm
    }
 	& $-3$	& $-2$	& $-1$	& $0$	& $1$	& $2$	& $3$	\\
$-8$	&  	& $2$	&  	&  	&  	&  	&  	\\
$-7$	&  	&  	& $8$	&  	&  	&  	&  	\\
$-6$	& $2$	&  	&  	& $12$	&  	&  	&  	\\
$-5$	& $2$	& $8$	&  	&  	& $8$	&  	&  	\\
$-4$	&  	& $6$	& $14$	&  	&  	& $2$	&  	\\
$-3$	&  	& \SetCell{red!25} 	& $9$	& $18$	&  	&  	&  	\\
$-2$	&  	&  	&  	& $11$	& $14$	&  	&  	\\
$-1$	&  	&  	& \SetCell{red!25} 	&  	& $9$	& $8$	&  	\\
$0$	&  	&  	&  	& \SetCell{green!25}$2$	&  	& $6$	& $2$	\\
$1$	&  	&  	&  	&  	& \SetCell{red!25} 	&  	& $2$	\\
\end{tblr}
}\hfill
\scalebox{.75}{%
\begin{tblr}[b]{
        hlines={0.7pt, solid}, vlines={0.7pt, solid},
        hline{3-Y} = {0.3pt, dashed}, vline{3-Y} = {0.3pt, dashed},
        rowspec={*{18}c}, rows={9mm}, columns={9mm},
        rowsep=0mm, colsep=0mm
    }
 	& $-4$	& $-3$	& $-2$	& $-1$	& $0$	& $1$	& $2$	& $3$	& $4$	\\
$-15$	&  	& $1$	&  	&  	&  	&  	&  	&  	&  	\\
$-14$	&  	&  	& $6$	&  	&  	&  	&  	&  	&  	\\
$-13$	&  	&  	&  	& $15$	&  	&  	&  	&  	&  	\\
$-12$	&  	&  	&  	&  	& $20$	&  	&  	&  	&  	\\
$-11$	& $1$	&  	&  	&  	&  	& $15$	&  	&  	&  	\\
$-10$	&  	& $8$	&  	&  	&  	&  	& $6$	&  	&  	\\
$-9$	&  	&  	& $24$	&  	&  	&  	&  	& $1$	&  	\\
$-8$	& $4$	&  	&  	& $40$	&  	&  	&  	&  	&  	\\
$-7$	& $3$	& $19$	&  	&  	& $48$	&  	&  	&  	&  	\\
$-6$	&  	& $12$	& $44$	&  	& $2$	& $40$	&  	&  	&  	\\
$-5$	&  	&  \SetCell{red!25}	& $26$	& $69$	&  	&  	& $24$	&  	&  	\\
$-4$	&  	&  	&  	& $40$	& $78$	&  	&  	& $8$	&  	\\
$-3$	&  	&  	&  \SetCell{red!25}	&  	& $47$	& $69$	&  	&  	& $1$	\\
$-2$	&  	&  	&  	&  	& $2$	& $40$	& $44$	&  	&  	\\
$-1$	&  	&  	&  	&  \SetCell{red!25}	&  	&  	& $26$	& $19$	&  	\\
$0$	&  	&  	&  	&  	& \SetCell{green!25}$2$	&  	&  	& $12$	& $4$	\\
$1$	&  	&  	&  	&  	&  	& \SetCell{red!25} 	&  	&  	& $3$	\\
\end{tblr}}%
\caption{Homologies of $\MM_3$ and~$\MM_4$. See \cref{table:hom1_2} for explanations.}%
\label{table:hom3_4}
\end{table}

\begin{table}[p]
\scalebox{.525}{
\begin{tblr}[b]{
        hlines={0.7pt, solid}, vlines={0.7pt, solid},
        hline{3-Y} = {0.3pt, dashed}, vline{3-Y} = {0.3pt, dashed},
        rowspec={*{26}c}, rows={8.8mm}, columns={8.8mm},
        rowsep=0mm, colsep=0mm
    }
 	& $-5$	& $-4$	& $-3$	& $-2$	& $-1$	& $0$	& $1$	& $2$	& $3$	& $4$	& $5$	\\
$-23$	&  	&  	& $2$	&  	&  	&  	&  	&  	&  	&  	&  	\\
$-22$	&  	&  	&  	& $12$	&  	&  	&  	&  	&  	&  	&  	\\
$-21$	&  	&  	&  	&  	& $30$	&  	&  	&  	&  	&  	&  	\\
$-20$	&  	&  	&  	&  	&  	& $40$	&  	&  	&  	&  	&  	\\
$-19$	&  	&  	&  	&  	&  	&  	& $30$	&  	&  	&  	&  	\\
$-18$	&  	&  	&  	&  	&  	&  	&  	& $12$	&  	&  	&  	\\
$-17$	&  	& $4$	&  	&  	&  	&  	&  	&  	& $2$	&  	&  	\\
$-16$	&  	&  	& $24$	&  	&  	&  	&  	&  	&  	&  	&  	\\
$-15$	&  	&  	&  	& $64$	&  	&  	&  	&  	&  	&  	&  	\\
$-14$	&  	&  	&  	&  	& $104$	&  	&  	&  	&  	&  	&  	\\
$-13$	& $2$	&  	&  	&  	&  	& $122$	&  	&  	&  	&  	&  	\\
$-12$	&  	& $20$	&  	&  	&  	& $2$	& $104$	&  	&  	&  	&  	\\
$-11$	&  	&  	& $74$	&  	&  	&  	&  	& $64$	&  	&  	&  	\\
$-10$	& $8$	&  	&  	& $154$	&  	&  	&  	&  	& $24$	&  	&  	\\
$-9$	& $6$	& $44$	&  	&  	& $220$	&  	&  	&  	&  	& $4$	&  	\\
$-8$	&  	& $28$	& $122$	&  	&  	& $242$	&  	&  	&  	&  	&  	\\
$-7$	&  	&  \SetCell{red!25}	& $70$	& $232$	&  	&  	& $220$	&  	&  	&  	&  	\\
$-6$	&  	&  	&  	& $130$	& $334$	& $2$	&  	& $154$	&  	&  	&  	\\
$-5$	&  	&  	&  \SetCell{red!25}	&  	& $193$	& $378$	&  	&  	& $74$	&  	&  	\\
$-4$	&  	&  	&  	&  	&  	& $221$	& $334$	&  	&  	& $20$	&  	\\
$-3$	&  	&  	&  	&  \SetCell{red!25}	& 	&  	& $193$	& $232$	&  	&  	& $2$	\\
$-2$	&  	&  	&  	&  	& 	& $2$	&  	& $130$	& $122$	&  	&  	\\
$-1$	&  	&  	&  	&  	& \SetCell{red!25}	&  	&  	&  	& $70$	& $44$	&  	\\
$0$	&  	&  	&  	&  	&  	& \SetCell{green!25}$2$	&  	&  	&  	& $28$	& $8$	\\
$1$	&  	&  	&  	&  	&  	&  	& \SetCell{red!25} 	&  	&  	&  	& $6$	\\
\end{tblr}}%
\hfill
\scalebox{.525}{%
\begin{tblr}[b]{hlines={0.7pt, solid}, vlines={0.7pt, solid}, hline{3-Y} = {0.3pt, dashed}, vline{3-Y} = {0.3pt, dashed}, rowspec={*{37}c}, rows={8.8mm}, columns={8.8mm}, rowsep=0mm, colsep=0mm}
 	& $-6$	& $-5$	& $-4$	& $-3$	& $-2$	& $-1$	& $0$	& $1$	& $2$	& $3$	& $4$	& $5$	& $6$	\\
$-34$	&  	&  	& $1$	&  	&  	&  	&  	&  	&  	&  	&  	&  	&  	\\
$-33$	&  	&  	&  	& $8$	&  	&  	&  	&  	&  	&  	&  	&  	&  	\\
$-32$	&  	&  	&  	&  	& $28$	&  	&  	&  	&  	&  	&  	&  	&  	\\
$-31$	&  	&  	&  	&  	&  	& $56$	&  	&  	&  	&  	&  	&  	&  	\\
$-30$	&  	&  	&  	&  	&  	&  	& $70$	&  	&  	&  	&  	&  	&  	\\
$-29$	&  	&  	&  	&  	&  	&  	&  	& $56$	&  	&  	&  	&  	&  	\\
$-28$	&  	&  	&  	&  	&  	&  	&  	&  	& $28$	&  	&  	&  	&  	\\
$-27$	&  	&  	&  	&  	&  	&  	&  	&  	&  	& $8$	&  	&  	&  	\\
$-26$	&  	& $2$	&  	&  	&  	&  	&  	&  	&  	&  	& $1$	&  	&  	\\
$-25$	&  	&  	& $18$	&  	&  	&  	&  	&  	&  	&  	&  	&  	&  	\\
$-24$	&  	&  	&  	& $70$	&  	&  	&  	&  	&  	&  	&  	&  	&  	\\
$-23$	&  	&  	&  	&  	& $160$	&  	&  	&  	&  	&  	&  	&  	&  	\\
$-22$	&  	&  	&  	&  	&  	& $248$	&  	&  	&  	&  	&  	&  	&  	\\
$-21$	&  	&  	&  	&  	&  	&  	& $286$	&  	&  	&  	&  	&  	&  	\\
$-20$	& $1$	&  	&  	&  	&  	&  	& $2$	& $248$	&  	&  	&  	&  	&  	\\
$-19$	&  	& $16$	&  	&  	&  	&  	&  	&  	& $160$	&  	&  	&  	&  	\\
$-18$	&  	&  	& $83$	&  	&  	&  	&  	&  	&  	& $70$	&  	&  	&  	\\
$-17$	&  	&  	&  	& $236$	&  	&  	&  	&  	&  	&  	& $18$	&  	&  	\\
$-16$	&  	&  	&  	&  	& $451$	&  	&  	&  	&  	&  	&  	& $2$	&  	\\
$-15$	& $6$	&  	&  	&  	&  	& $644$	&  	&  	&  	&  	&  	&  	&  	\\
$-14$	&  	& $52$	&  	&  	&  	&  	& $720$	&  	&  	&  	&  	&  	&  	\\
$-13$	&  	&  	& $200$	&  	&  	&  	&  	& $644$	&  	&  	&  	&  	&  	\\
$-12$	& $15$	&  	&  	& $484$	&  	&  	& $2$	&  	& $451$	&  	&  	&  	&  	\\
$-11$	& $10$	& $96$	&  	&  	& $858$	&  	&  	&  	&  	& $236$	&  	&  	&  	\\
$-10$	&  	& $58$	& $316$	&  	&  	& $1192$	&  	&  	&  	&  	& $83$	&  	&  	\\
$-9$	&  	&  \SetCell{red!25}	& $182$	& $716$	&  	&  	& $1330$	&  	&  	&  	&  	& $16$	&  	\\
$-8$	&  	&  	&  	& $406$	& $1236$	&  	&  	& $1192$	&  	&  	&  	&  	& $1$	\\
$-7$	&  	&  	&  \SetCell{red!25}	&  	& $697$	& $1696$	&  	&  	& $858$	&  	&  	&  	&  	\\
$-6$	&  	&  	&  	&  	&  	& $950$	& $1882$	&  	&  	& $484$	&  	&  	&  	\\
$-5$	&  	&  	&  	&  \SetCell{red!25}	&  	&  	& $1051$	& $1696$	&  	&  	& $200$	&  	&  	\\
$-4$	&  	&  	&  	&  	&  	&  	&  	& $950$	& $1236$	&  	&  	& $52$	&  	\\
$-3$	&  	&  	&  	&  	&  \SetCell{red!25}	&  	&  	&  	& $697$	& $716$	&  	&  	& $6$	\\
$-2$	&  	&  	&  	&  	&  	&  	& $2$	&  	&  	& $406$	& $316$	&  	&  	\\
$-1$	&  	&  	&  	&  	&  	& \SetCell{red!25} 	&  	&  	&  	&  	& $182$	& $96$	&  	\\
$0$	&  	&  	&  	&  	&  	&  	& \SetCell{green!25}$2$	&  	&  	&  	&  	& $58$	& $15$	\\
$1$	&  	&  	&  	&  	&  	&  	&  	&  \SetCell{red!25}	&  	&  	&  	&  	& $10$	\\
\end{tblr}}
\caption{Homology of $\MM_5$ and~$\MM_6$. See \cref{table:hom1_2} for explanations.}%
\label{table:hom5_6}
\end{table}
\end{document}

%% file: preamble.tex
\usepackage[latin1]{inputenc}
\usepackage{amsmath,amssymb,thmtools,graphicx,url,microtype,caption,xcolor,tabularray}
\usepackage[bookmarks=true,pdfencoding=auto,hidelinks]{hyperref}
\hypersetup{pdfauthor={\authors},pdftitle={\shorttitle}}
\usepackage[nameinlink]{cleveref}

\let\cref\Cref
\crefname{subsection}{section}{sections}
\Crefname{subsection}{Section}{Sections}
\Crefname{enumi}{}{}
\crefname{equation}{}{}
\newcommand{\myemail}[1]{\href{mailto:#1}{#1}}

\newcommand{\qua}{\hskip 0.4em \ignorespaces}
\def\arxiv#1{\relax\ifhmode\unskip\qua\fi
\href{http://arxiv.org/abs/#1}%
{\tt arXiv:\penalty -100\unskip#1}}

\def\MR#1{\relax\ifhmode\unskip\qua\fi
\href{https://mathscinet.ams.org/mathscinet-getitem?mr=#1}{\tt MR#1}}
\def\ZB#1{\relax\ifhmode\unskip\qua\fi
\href{https://zbmath.org/?q=an:#1}{\tt Zbl\:#1}}
\def\xox#1{\csname xx#1\endcsname}

\renewenvironment{thebibliography}[1]{
  \begin{oldthebibliography}{#1}\small
    \setlength{\itemsep}{.5ex}
    \setlength{\parskip}{0em}
}
{
  \end{oldthebibliography}
}

\let\stdthebibliography\thebibliography
\let\stdendthebibliography\endthebibliography
\renewenvironment*{thebibliography}[1]{%
  \stdthebibliography{BBP+25}}
  {\stdendthebibliography}

\declaretheorem{lemma}
\newtheorem{theorem}[lemma]{Theorem}

\newtheorem{proposition}[lemma]{Proposition}
\newtheorem{conjecture}[lemma]{Conjecture}
\theoremstyle{definition}

\newtheorem{remark}[lemma]{Remark}

\DeclareMathOperator{\Gr}{Gr}

\DeclareMathOperator{\HFK}{HFK}
\DeclareMathOperator{\CFK}{CFK}

\graphicspath{{figures/}}